\documentclass[11pt]{article}
 \usepackage{epsfig}
 \usepackage{amssymb,amsmath,amsthm,amscd}
\usepackage{latexsym}
\usepackage{showlabels}
\usepackage{graphicx}
\usepackage{color}
\usepackage{subfigure}
 \theoremstyle{theorem}
 \newtheorem{thm}{Theorem}[section]
 \newtheorem{lemma}[thm]{Lemma}
  
 \newtheorem{definition}[thm]{Definition}
 
 \newtheorem{cor}[thm]{Corollary}
  \newtheorem{remark}[thm]{Remark}
 \title{Local first integrals for stochastic differential equations}
\author{
  Kaiyin Huang\footnote{School of Mathematics, Sichuan University, Chengdu 630065,  China ({\tt huangky@scu.edu.cn}).},
          \;
          Wenlei Li\footnote{College of Mathematics, Jilin University,
Changchun, Jilin 130012,  China ({\tt  lwlei@jlu.edu.cn}).},
          \;
         Shaoyun Shi\footnote{College of Mathematics, Jilin University,
Changchun, Jilin 130012, China ({\tt  shisy@jlu.edu.cn}).},
          \;
  and Zhiguo Xu\footnote{College of Mathematics, Jilin University,
Changchun, Jilin 130012, China ({\tt xuzg2014@jlu.edu.cn}).}. 
}
\date{\empty}

\begin{document}
 \maketitle

 \begin{abstract}
Poincar\'{e}'s classical results [H. Poincar\'{e}, Sur l'int\'{e}gration des \'{e}quations
diff\'{e}rentielles du premier order et du premier degr\'{e} I and II, Rend. Circ. Mat. Palermo 5 (1891) 161-191;
11 (1897) 193-239] first provide a link between the existence of analytic first integrals and the resonant relations for
analytic dynamical systems. In this paper, we show that by appropriately selecting the definition of the
stochastic local first integrals, we are able to obtain the stochastic version  of Poincar\'e non-integrability theorem. More specifically, we introduce two definitions of local first integrals for stochastic differential equations(SDEs)
in the sense of probability one and  expectation,  respectively. We present the necessary conditions for
 the existence of functionally independent analytic or rational first integrals of SDEs via the resonances.
  We also  show that for given integrable ordinary differential equations with some nondegeneracy conditions,  there exists a linear stochastic perturbation such that the
corresponding disturbed SDEs have no any analytic first integrals.   Some examples are given to illustrate our results.

 \end{abstract}
 \newpage

\section{Introduction}\label{sec1}

In 1891, Poincar\'e \cite{poincare1891}  provided a criteria to study the non-existence of analytic first integrals for analytic differential systems via resonant relations. He obtained the following classical result (for a proof, see \cite{furta1996}).

\begin{thm}[Poincar\'e non-integrability theorem]\label{poincare}
 Assume that the analytic differential system
 \begin{align}\label{1.1}
  \frac{\mathrm{d}x}{\mathrm{d}t}=f(x),~~~x\in\mathbb{C}^n
 \end{align}
has a singularity at $x=0$, i.e., $f(0)=0$. If the eigenvalues of the Jacobian matrix $Df(0)$ at $x=0$ do not satisfy any $\mathbb{Z^+}$-resonant conditions, then system $(\ref{1.1})$ has no analytic first integrals in a neighborhood of the origin.
\end{thm}

 Let us recall some terminologies. Denote by $\mathrm{U}$ an open set in $\mathbb{C}^n$. A non-constant function $\Phi(x):\mathrm{U}\rightarrow \mathbb{C}$ is called a first
integral of (\ref{1.1}) if it is constant along any solution of (\ref{1.1}). Further, if it is an analytic or rational function with
respect to $x$, then it is called an analytic or rational first integral of system (\ref{1.1}), respectively. We call that the eigenvalues $\mathbf{\lambda} =(\lambda_1,\cdots\lambda_n)$ of a $n\times n$ matrix satisfy a $\mathbb{Z^+}$-resonant condition if
$
\langle \mathbf{\lambda}, \mathbf{k} \rangle=0
$
for some $\mathbf{k}\in(\mathbb{Z}^+)^n/\{0\}$,
where $\mathbb{Z^+}=\mathbb{N}\cup\{0\}$ and $\langle \cdot,\cdot\rangle$ denotes the inner product of two vectors in $\mathbb{C}^n$. 

The main idea of Poincar\'{e}'s result is that he found a relation between the existence of analytic first integrals and the properties of variation equations
along the equilibrium solution, namely, the resonance of eigenvalues of the matrix $Df(0)$. Following Poincar\'{e}'s idea, many scholars have devoted themselves in studying non-integrability and partial integrability for nonlinear systems of differential equations, and there have appeared large numbers of results on this issue, see for instance
Yoshida \cite{yoshida1983},  Furta \cite{furta1996}, Goriely \cite{goriely1996}, Shi \cite{shi2001},
Zhang \cite{zhang2003}, Kwek \emph{et al.}\cite{kwek2003},  Li \emph{et al.}\cite{li2003}, Chen \emph{et al.}\cite{chen2008jde}, Wang \emph{et al.}\cite{cong2011jde}, Romanovski \emph{et al.}\cite{r2014jde}, Du \emph{et al.}\cite{du2016jde} and the references therein.

In recent years, stochastic differential equations(SDEs) appear in many branches of science such as biology, epidemiology, mechanics and economics, and are concerned by more and more scholars. Many dynamical features of SDEs has been also studied such as random periodic solution \cite{zhao2009}, almost automorphic solution \cite{liuzx1}, stability \cite{mao1994}, attractor-repeller pair \cite{liuzx2}, Conley index \cite{liuzx3}, center manifold \cite{boxler1989}, random invariant manifolds \cite{lukn1} and normal forms \cite{lukn2,lukn3}. The aim of this paper is to investigate the first integrals for SDEs and to generalize the Poincar\'e non-integrability theorem from ODEs to SDEs.


Consider the following SDEs of It\^o type
\begin{equation}\label{e1}
 \mathrm{d}X_t=f(X_t)\mathrm{d}t+\sum\limits_{i = 1}^m {{g_i}(X_t)\mathrm{d}B^i_t},
\end{equation}
where  $X_t=(X^1_t,\cdots,X^n_t)^\mathrm{T}\in \mathbb{C}^n$, $f(X_t)=\left(f^1(X_t),\cdots,f^n(X_t)\right)^\mathrm{T}$ and $g_i(X_t)=\left({g_i^1(X_t)},...,g_i^n(X_t)\right)^\mathrm{T}, i=1,\cdots,m,$ are $n$-dimensional vector-valued analytic functions, and $B^i_t$ are jointly independent one-dimensional real Winner processes defined on a probability space $(\Omega,\mathcal{F}, \mathbb{P})$ with a non-decreasing
family of $\sigma$-algebra $\mathcal{F}_t \subset \mathcal{F}$, $t>0$. We remark that SDEs (\ref{e1}) of It\^o type can be  written equivalently as the following
SDEs of Stratonovich type as
\begin{equation}\label{e2}
 \mathrm{d}X_t=\big(f(X_t)-\frac{1}{2}\sum\limits_{i = 1}^m {D{g_i}(X_t)}\cdot g_i(X_t)\big)\mathrm{d}t+\sum\limits_{i = 1}^m {{g_i}(X_t)\circ \mathrm{d}B^i_t},
\end{equation}
see \cite{cong} for details. In this paper, we deal with local first integrals of SDEs in the frame of It\^o type, and our results obtained for SDEs (\ref{e1}) of It\^o type can be applied to the corresponding SDEs of Stratonovich type (\ref{e2}).

To our knowledge, the definition of first integrals for SDEs of It\^o type can be traced back to Doobko \cite{dubko1978}. Then Kulinich \cite{kulinich2008} considered first integrals for homogeneous stochastic differential equations with jump. For SDEs of Stratonovich type, Misawa \cite{Misawa1994,Misawa1999} studied the relation between first integrals and symmetry of SDEs of Stratonovich type. Hong \emph{et al.} \cite{hong} constructed an equivalent skew gradient form of SDEs and gave some conserved numerical methods when they considered SDEs with a first integral. Similar to ODEs, Zung in \cite{zung} showed that the first integrals of SDEs play a key role in the reduction theory of SDEs. For more details about this topic,
we refer the reader to \cite{M.thieullen1997,cresson,Cresson2007,Zambrini} and the references therein.

Compared to the definition of local first integrals for ODEs, to define the local first integral of SDEs, we replace the time $t$ by a random time $\tau=\tau(\omega)$ of a more
general type called stopping time \cite{bb1989}.
Let $\mathrm{U}$ be an open subset of $\mathbb{C}^n$ and $\tau_{U}$ be the first exit time for a solution $x(t)$ of system (\ref{e1}) to leave $\mathrm{U}$, i.e., $\tau_{\mathrm{U}}=\text {inf}\{t>0, X_t\notin \mathrm{U} \}$.
The following facts on stopping time are well known.
\begin{itemize}
\item For any fixed $t>0$, $\tau(\omega)= t$ is trivially a stopping time.

\item  The first exit time  $\tau_{\mathrm{U}}=\text {inf}\{t>0, X_t\notin \mathrm{U} \}$ is a stopping time.

\item  Let $\tau$ and $\sigma$ be two stopping times. Then $\tau\wedge\sigma:=\mathrm{min}\{\tau,\sigma\}$, $\tau\vee\sigma=\mathrm{max}\{\tau,\sigma\}$ are also stopping times.
\end{itemize}

\begin{definition}\label{deff1}
A non-constant function $\Phi(x):\mathrm{U}\rightarrow \mathbb{C}$ is called a strong first integral of $(\ref{e1})$ if for any solution $X_t$ of $(\ref{e1})$ with the deterministic initial value condition $X_0=x_0\in \mathrm{U}$, we have
\begin{align}
\mathbb{P}\big(\Phi(X_{\tau\wedge \tau_\mathrm{U}})=\Phi(x_0) \big)=1,~~\text{\rm for any stopping time}~~\tau.\label{def1}
\end{align}
\end{definition}
\begin{definition}\label{deff2}
A non-constant function $\Psi(x):\mathrm{U}\rightarrow \mathbb{C}$ is called a weak first integral of $(\ref{e1})$ if
for any solution $X_t$ of $(\ref{e1})$ with the deterministic initial value condition $X_0=x_0\in \mathrm{U}$, we have
\begin{align}
\mathbb{E}\big(\Psi(X_{\tau\wedge \tau_\mathrm{U}})\big)=x_0, ~~\text{\rm for any stopping time}~~\tau.\label{def2}
\end{align}
\end{definition}
\begin{remark}
By definitions \ref{deff1},\ref{deff2}, we can see easily that a strong first integral of SDEs $(\ref{e1})$ must be a weak first integral. In practical application, many stochastic models do not
admit any strong first integrals but have a weak first integral, for example see Example 1 in Sect.3.
\end{remark}

\begin{remark}
When noises $g_i$ is turned off, i.e., $g_i\equiv0$, the above definitions coincide with
the first integral in the deterministic case.
\end{remark}

In the present paper, we first give an equivalent characterization of strong(weak) first integrals for SDEs (\ref{e1}), which is more convenient to check whether given functions are strong(weak) first integrals for (\ref{e1}). In fact, we show that a strong first integral of (\ref{e1}) should be a common first integral of several ODEs and a weak first integral should satisfy a partial differential equation, which is known as the characteristic operator of (\ref{e1}). In addition, we try to get some necessary conditions of the existence of strong first integrals and weak first integrals of (\ref{e1}), respectively, which can be regarded as
a criteria for the non-existence of first integrals
for SDEs (\ref{e1}). For strong first integrals, the corresponding
result is straightforward by using some results of deterministic differential equations. For weak first integrals, under certain assumptions, a generalization of Poincar\'e non-integrability theorem is proposed via resonance conditions. Then, observing that many stochastic systems emerge
from a deterministic system by a stochastic perturbation, we also consider two special cases for which the stochastic terms of SDEs (\ref{e1}) is a higher order perturbation and a linear perturbation, respectively.

This paper is organized as follows. In Sect.2, we provide some analytic  characterizations of both strong first integrals and weak first integrals for SDEs (\ref{e1}), and give the necessary conditions for SDEs to have strong first integrals or weak first integrals, respectively.
 Some examples are given in Sect.3 to demonstrate our results. Finally, conclusions and
discussion of future directions are given in Sect.4.

\section{The main results}

Firstly, we provide the analytic characterization of local first integrals for stochastic differential equations.
\begin{thm}\label{the1}
Let $\Phi(x):\mathrm{U}\rightarrow \mathbb{C}$ be a twice continuous differentiable function. Then the following statements hold.

{\bf(a)} $\Phi(x)$ is a strong first integral of (\ref{e1}) if and only if
\begin{equation}\label{eqv1}
\langle\nabla\Phi(x), f(x)-\frac{1}{2}\sum\limits_{i = 1}^m {D{g_i}(x)\cdot g_i}\rangle\equiv 0,
\end{equation}
\begin{equation}\label{eqv22}
\langle\nabla\Phi(x),  g_i(x)\rangle\equiv0,~~~~\text{for}~~i=1,\cdots,m,
\end{equation}
where $\nabla\Phi$ is the gradient of $\Phi(x)$.

{\bf(b)} $\Phi(x)$ is a weak first integral of (\ref{e1}) if and only if
\begin{equation}\label{eqv3}
\langle\nabla\Phi(x), f(x)\rangle+\frac{1}{2}\sum\limits_{i = 1}^m {g_i^{T}(x)\nabla^2\Phi(x) g_i(x)}\equiv 0,
\end{equation}
where $\nabla^2\Phi(x)=\nabla\cdot\nabla^{T}\Phi(x)$ is the Hessian matrix of $\Phi(x)$.
\end{thm}

To prove Theorem \ref{the1}, we need the following Lemma.
\begin{lemma}\label{lemma1}
 Let $\Phi(x):\mathrm{U}\rightarrow \mathbb{C}$ be a twice continuous differentiable function which satisfies $(\ref{eqv22})$. Then $\Phi(x)$ satisfies $(\ref{eqv1})$ if and only if it satisfies $(\ref{eqv3})$.
\end{lemma}

\begin{proof} The result will be proved if we can show
$$
\sum\limits_{i = 1}^m {g_i^{T}(x)\nabla^2\Phi(x) g_i(x)}\equiv -\langle\nabla\Phi(x),\sum\limits_{i = 1}^m {D{g_i}(x)\cdot g_i}\rangle.
$$
In deed, differentiating both sides of (\ref{eqv22})
with respect to $x_j$ gives
\begin{align}\label{q1}
\sum\limits_{k = 1}^n {\frac{{\partial \Phi }}{{\partial {x_k}}}} \frac{{\partial g_i^k}}{{\partial {x_j}}} + \sum\limits_{k = 1}^n {\frac{{{\partial ^2}\Phi }}{{\partial {x_j}\partial {x_k}}}} g_i^k\equiv0,~~~~~j=1,\cdots,n.
\end{align}
Then, we have
\begin{align*}
\sum\limits_{i = 1}^m {g_i^{T}(x)\nabla^2\Phi(x) g_i(x)}&=\sum\limits_{i = 1}^m {\sum\limits_{k,j = 1}^n {\frac{{{\partial ^2}\Phi }}{{\partial {x_j}\partial x{}_k}}} } g_i^kg_i^j\\
&=\sum\limits_{i = 1}^m {\sum\limits_{j = 1}^n {g_i^j(\sum\limits_{k = 1}^n {\frac{{{\partial ^2}\Phi }}{{\partial {x_j}\partial x{}_k}}g_i^k)} } }\\
&=- \sum\limits_{i = 1}^m {\sum\limits_{j = 1}^n {g_i^j(\sum\limits_{k = 1}^n {\frac{{\partial \Phi }}{{\partial {x_k}}}\frac{{\partial g_i^k}}{{\partial {x_j}}})} } }~~~~\text{using (\ref{q1})}\\
&= - \sum\limits_{i = 1}^m {\sum\limits_{k = 1}^n {\frac{{\partial \Phi }}{{\partial {x_k}}}} } (\sum\limits_{j = 1}^n {\frac{{\partial g_i^k}}{{\partial {x_j}}}} g_i^j)\\
&=-\sum\limits_{i=1}^m{\langle\nabla\Phi(x),D{g_i}(x)\cdot g_i\rangle}\\
&=-\langle\nabla\Phi(x),\sum\limits_{i = 1}^m {D{g_i}(x)\cdot g_i}\rangle.
\end{align*}
This completes the proof of Lemma \ref{lemma1}.
\end{proof}

\noindent\textbf{Proof of Theorem \ref{the1}.}\  Let $X_t$ be the solution of (\ref{e1}) with the deterministic initial value condition $X_0=x_0\in \mathrm{U}$. Due to the It\^o formula, we have
\begin{align}\label{ss1}
\mathrm{d}\Phi\big(X_t\big)=\Big(\langle\nabla\Phi(X_t), f(X_t)\rangle&+\frac{1}{2}\sum\limits_{i = 1}^m {g_i^{T}(X_t)\nabla^2\Phi(X_t) g_i(X_t)}\Big)\mathrm{d}t\notag\\
&+\sum\limits_{i=1}^m{\big\langle \nabla\Phi(X_t), g_i(X_t)\big\rangle \mathrm{d}B_t^i}.
\end{align}
Set
$$
G(x)=\langle\nabla\Phi(x), f(x)\rangle+\frac{1}{2}\sum\limits_{i = 1}^m {g_i^{T}(x)\nabla^2\Phi(x) g_i(x)},
$$
then (\ref{ss1}) can be transformed into an equivalent integral equation
\begin{align}
\Phi(X_t)=\Phi(x_0)+\int_0^t {G({X_t})\mathrm{d}t} +\sum\limits_{i = 1}^m{\int_0^t{\big\langle \nabla\Phi(X_t), g_i(X_t)\big\rangle \mathrm{d}B_t^i}}.\label{i3}
\end{align}

{\bf(a)} It follows from Lemma \ref{lemma1} that we need only to show $\Phi(x)$ is a strong first integral of (\ref{e1}) if and only if (\ref{eqv22}) and (\ref{eqv3}) hold.

If $\Phi(x)$ is a strong first integral of (\ref{e1}), then by definition, we have
\begin{align}\label{gg}
\int_0^{\tau\wedge\tau_U} {G({X_t})\mathrm{d}t} +\sum\limits_{i = 1}^m{\int_0^{\tau\wedge\tau_U}{\big\langle \nabla\Phi(X_t), g_i(X_t)\big\rangle \mathrm{d}B_t^i}}=0,~~\text{almost~everywhere}
\end{align}
for any solution $X_t$ of (\ref{e1}) with the deterministic initial value condition $X_0=x_0\in U$, and any stopping time $\tau$.
Taking
the expectations on both sides of (\ref{gg}),
\begin{align}\label{e0}
\mathbb{E}\left(\int_0^{\tau\wedge\tau_\mathrm{U}} {G({X_t})\mathrm{d}t}\right)=0.
\end{align}

We claim that $G(x)\equiv0$ for $x\in \mathrm{U}$. In fact, suppose $G(x_1)\neq0$ for some $x_1\in \mathrm{U}$,
without loss of generality, we can assume $G(x) > 0$ in a neighborhood $V\subset U$ of $x_1$.
Consider the solution $X_t$ of SDEs (\ref{e1}) with initial value $X_0=x_1$, then $G(X_t)\neq 0$ for $t\in [0, \tau_{V}) \subset[0, \tau_\mathrm{U})$, where $\tau_{V}=\text {inf}\{t>0, X_t\notin V \}$. For stopping time $\tau=\tau_V$, we obtain
$$
\int_0^{\tau_V\wedge\tau_\mathrm{U}} {G({X_t})\mathrm{d}t}\neq0,
$$
which implies a contradiction with (\ref{e0}). Therefore, $G(x)\equiv0$ in $U$, i.e., (\ref{eqv3}) holds and
\begin{align}\label{ee0}
\sum\limits_{i = 1}^m{\int_0^{\tau\wedge\tau_\mathrm{U}}{\big\langle \nabla\Phi(X_t), g_i(X_t)\big\rangle \mathrm{d}B_t^i}}=0,~~\text{almost~~everywhere.}
\end{align}

To prove (\ref{eqv22}), by (\ref{ee0}) and the It\^{o} isometry formula, we get
\begin{align*}
\mathbb{E}\Big(\int_0^{\tau\wedge \tau_\mathrm{U}}{\sum\limits_{i=1}^m{\big|\langle \nabla\Phi(X_t), g_i(X_t)\rangle\big|^2 \mathrm{d}t}}\Big)=\mathbb{E}\Big(\sum\limits_{i = 1}^m{\int_0^{\tau\wedge\tau_U}{\big\langle \nabla\Phi(X_t), g_i(X_t)\big\rangle \mathrm{d}B_t^i}}\Big)^2=0.
\end{align*}
\noindent Then, using the same argument, we see
$$
\sum\limits_{i=1}^m{\big|\langle \nabla\Phi(x), g_i(x)\rangle\big|^2}\equiv0,~~\text{i.e.,}~~\langle\nabla\Phi(x),  g_i(x)\rangle\equiv0,~~~~\text{for}~~i=1,\cdots,m.
$$

If (\ref{eqv22}) and (\ref{eqv3}) hold, then it is straightforward to see that $\Phi(x)$ is a strong first integral
by (\ref{i3}).

{\bf(b)} It follows from (\ref{i3}) that
$$
\mathbb{E}\big(\Phi(X_{\tau\wedge \tau_\mathrm{U}})\big)=\Phi(x_0)+\mathbb{E}\big(\int_0^{\tau\wedge\tau_\mathrm{U}} {G({X_t})\mathrm{d}t}\big),~~\text{for any stopping time}~\tau.
$$
Therefore $\Phi(x)$ is a weak first integral of SDEs (\ref{e1}) if and only if
$$
\mathbb{E}\big(\int_0^{\tau\wedge\tau_\mathrm{U}} {G({X_t})\mathrm{d}t}\big)\equiv0.
$$
The rest of the proof is the same as above.

\begin{remark}
The fact that strong first integrals should be weak first integrals coincides with
the fact that the equations $(\ref{eqv1})$ and $(\ref{eqv22})$ imply
the equation $(\ref{eqv3})$.
\end{remark}
\begin{remark}\label{remark1}
Thanks to Theorem \ref{the1}, the function $\Phi(x)$ is a strong first integral of SDEs $(\ref{e1})$ if and only if it is a common first integral of the following $m+1$ systems of differential equations.
\begin{align}
\frac{\mathrm{d}x}{\mathrm{d}t}&=f(x)-\frac{1}{2}\sum\limits_{i = 1}^m {D{g_i}(x)\cdot g_i(x)},\label{f}\\
\frac{\mathrm{d}x}{\mathrm{d}t}&= g_i(x)~~~~i=1,\cdots,m.\label{g}
\end{align}
\end{remark}

 Now, we investigate the maximum number of functionally independent analytic strong first integrals for SDEs (\ref{e1}). But before that, we introduce some notations. Let $\lambda=(\lambda_1,\cdots,\lambda_n)$ and $\mu^i=(\mu^i_{1},\cdots,\mu^i_{n})$ be
the eigenvalues of the matrices $Df(0)-\frac{1}{2}\sum\limits_{j=1}^m{Dg_i(0)^2}$ and $Dg_i(0)$, respectively.
Set
\begin{align*}
\mathcal{S}_0&=\{\mathbf{k}\in(\mathbb{Z}^+)^n: \langle \lambda, \mathbf{k} \rangle=0,~~\mathbf{k}\neq0,\},\\
\mathcal{S}_i&=\{\mathbf{k}\in(\mathbb{Z}^+)^n: \langle \mu^i, \mathbf{k} \rangle=0,~~\mathbf{k}\neq0\},~~~~i=1,\cdots,m,
\end{align*}
and denote by $s_j$ the rank of the set $\mathcal{S}_j$, i.e., the dimension of the linear space spanned by $\mathcal{S}_j$,
$j=0,1,\cdots,m$.

\begin{thm}\label{the2}
Assume that $f(0)=0$ and $g_i(0)=0,i=1,\cdots,m$. If SDEs $(\ref{e1})$ admits $s\triangleq min\{s_0, s_1, \cdots, s_m\}$ functionally independent analytic first integrals $\Phi^1(x),\cdots,\Phi^s(x)$, then there exists a smooth function $\mathcal{H}$ such that any other nontrivial analytic first integral $\Phi(x)$ of $(\ref{e1})$ must be an analytic function of $\Phi^1(x),\cdots,\Phi^s(x)$, i.e.,
$$
\Phi(x)=\mathcal{H}(\Phi^1(x),...,\Phi^s(x)).
$$
\end{thm}

\textbf{Proof.}\ Without loss of generality, we assume $s=s_0$. By Theorem \ref{the1}, we know that $\Phi^1(x),\cdots,\Phi^s(x)$ and $\Phi(x)$
are also analytic first integrals of  system (\ref{f}). Then, for system (\ref{f}), it follows from Theorem 1 in \cite{zhang2003}
that there exists an analytic function $\mathcal{H}$ such that $
\Phi(x)=\mathcal{H}(\Phi^1(x),\cdots,\Phi^s(x)).
$
Obviously, the function $\mathcal{H}(\Phi^1(x),\cdots,\Phi^s(x))$ is also a common analytic first integral of the rest of $m$
systems (\ref{g}), i.e., an analytic strong first integral of (\ref{e1}).

We have the following simple conclusions.
\begin{cor}\label{coro2}
Assume that $f(0)=0$ and $g_i(0)=0,i=1,\cdots,m$. Then the number of functionally independent analytic strong first integrals in a neighborhood of the origin is less than or equal to $s\triangleq min\{s_0, s_1, \cdots, s_m\}$.
\end{cor}

\begin{cor}\label{coro}
Assume that $f(0)=0$, $g_i(0)=0,i=1,\cdots,m$. If  the eigenvalues of one of   $m+1$ matrices $Df(0)-\frac{1}{2}\sum\limits_{j=1}^m{Dg_i(0)^2}$ and $Dg_i(0),i=1,\cdots,m$  don't satisfy any $\mathbb{Z}^{+}$-resonant condition, then SDEs $(\ref{e1})$ don't have any analytic strong first integral in a neighborhood of the origin.
\end{cor}

\begin{remark}\label{rr}
Theorem \ref{the2} and Corollary \ref{coro2}, \ref{coro} are also true for the case of rational first integrals or Laurent polynomial first integrals of SDEs $(\ref{e1})$,
the major change being the substitution of $\mathbb{Z}^{+}$-resonant for $\mathbb{Z}$-resonant, for more details see \cite{shi2007,shi2006}.
\end{remark}

In what follows, we turn to consider the non-existence of weak first integrals of SDEs (\ref{e1}).
 Let the following hypothesis hold:

${\mathbf(\rm{H}1)}$ The matrices $Df(0)$, $Dg_1(0),\cdots, Dg_m(0)$ are simultaneous diagonalizable, namely, there exists a nonsingular matrix $T$ such that
\begin{align*}
T^{-1}Df(0)T&\triangleq\Lambda_0=\mathrm{diag}(\mu^0_1,\cdots,\mu^0_n),\\
T^{-1}Dg_i(0)T&\triangleq\Lambda_i=\mathrm{diag}(\mu^i_{1},\cdots,\mu^i_{n})~~~~\text{for}~~i=1,\cdots,m.
\end{align*}

Under the transformation $X_t=\varepsilon TY_t$, by It\^o formula, (\ref{e1}) can be rewritten as
\begin{equation}\label{e4}
\mathrm{d}Y_t=(\Lambda_0 Y_t+\varepsilon f_2(Y_t,\varepsilon))\mathrm{d}t+\sum\limits_{i = 1}^m {(\Lambda_i Y_t+\varepsilon g_{i2}(Y_t,\varepsilon))\mathrm{d}B^i_t},
\end{equation}
where $f_2(y,\epsilon)=O(\left| y \right|^2)), g_i(y,\varepsilon)=O(\left| y \right|^2)), i=1,\cdots,m$. Obviously, $\mu^0=(\mu^0_1,\cdots$, $\mu^0_n)$ and $\mu^i=(\mu^i_{1},\cdots,\mu^i_{n})$ are the eigenvalues of $Df(0)$ and $Dg_i(0)$, respectively.
Set $
\lambda_i=\mu^0_i-\frac{1}{2}\sum\limits_{j=1}^m{(\mu^i_j)^2}, i=1,\cdots,m.
$
Then $\lambda=(\lambda_1,\cdots,\lambda_n)$ are the eigenvalues of $Df(0)-\frac{1}{2}\sum\limits_{j=1}^m{Dg_i(0)^2}$.

\begin{thm}\label{the3}
Assume $(\rm{H}1)$ holds and $f(0)=0$, $g_i(0)=0,i=1\cdots,m$. If
\begin{align}
  \langle \lambda, \mathbf{k}\rangle+\frac{1}{2} \sum\limits_{i = 1}^m{|\langle \mu^i, \mathbf{k}\rangle|^2}\neq0
\end{align}
for any $\mathbf{k}\in (\mathbb{Z^+})^n\backslash \{\mathbf{0}\}$, then SDEs $(\ref{e1})$ don't admit any analytic weak first integral in a neighborhood of the origin.
\end{thm}

\begin{proof}
Assume that (\ref{e1}) has a weak analytic first integral
$$
\Phi(x)=\Phi_l(x)+\Phi_{l+1}(x)+\cdots,
$$
where $l,k\in \mathbb{N}\cup\{0\}$, $\Phi_i(x)$ are homogeneous polynomials of degree $i$, respectively. Without loss of generality, we assume that $\Phi_l(x)$
is not a constant. Obviously,
\begin{align}\label{psi}
\Psi(y,\varepsilon)=\Phi(\varepsilon Ty)=\varepsilon^{l-m}\big(\Phi_l(y)+\varepsilon \Phi_{l+1}(y)+\cdots\big)
\end{align}
is a weak analytic first integral of (\ref{e4}).
By Theorem \ref{the1}, we obtain
\begin{align}\label{eqv8888}
\langle\nabla\Psi(y), \Lambda_0 y+\varepsilon f_2(y,\varepsilon)\rangle+\frac{1}{2}\sum\limits_{i = 1}^m {(\Lambda_i y+\varepsilon g_{i2}(y,\varepsilon))^{T}\cdot\nabla^2\Psi \cdot(\Lambda_i y+\varepsilon g_{i2}(y,\varepsilon))}\equiv 0.
\end{align}
Substituting (\ref{psi}) into (\ref{eqv8888}) and equating the terms of the lowest order with respect to $\varepsilon $, one can get
\begin{align}\label{eqv7}
\langle\nabla\Psi_l(y), \Lambda_0 y\rangle+\frac{1}{2}\sum\limits_{i = 1}^m {(\Lambda_i y)^{T}\cdot\nabla^2\Psi_l \cdot(\Lambda_i y)}\equiv 0.
\end{align}
On the other hand, $\Psi_l(y)$ can be rewrite as a sum of elementary monomials
\begin{align*}
\Psi_l(y)=\sum\limits_{{l_1} +\cdots + {l_n} =l } {{\Psi_{{l_1}\cdots{l_n}}}y_1^{{l_1}}\cdots y_n^{{l_n}}},~~\Psi_{{l_1}...{l_n}}\in\mathbb{C}.
\end{align*}
By simple calculations, we have
\begin{align}
\langle\nabla\Psi_l(y), \Lambda_0 y\rangle&=\sum\limits_{{l_1} + \cdots + {l_n} =l } {{\Psi_{{l_1}\cdots{l_n}}}\left(\sum\limits_{i=1}^n{u^0_i l_i}\right)y_1^{{l_1}}\cdots y_n^{{l_n}}},\label{eqv5}\\
(\Lambda_i y)^{T}\cdot\nabla^2\Psi_l \cdot(\Lambda_i y)
&=\sum\limits_{{l_1} +\cdots+{l_n} = l} {\Psi_{{l_1}\cdots{l_n}}\left((\sum\limits_{j = 1}^n {{l_j}{\mu^i_j})^2 - \sum\limits_{j= 1}^n {{l_j}\mu^i_j\mu^i_j} } \right)} {y^{{l_1}}}{y^{{l_2}}}\cdots{y^{{l_n}}}.\label{eqv6}
\end{align}
Substituting (\ref{eqv5}) and (\ref{eqv6}) into (\ref{eqv7}) yields that
$$
\sum\limits_{{l_1} +\cdots+{l_n} = l} {\Psi_{{l_1}...{l_n}}\left(\sum\limits_{i = 1}^n {{\lambda _i}{l_i}}  + \frac{1}{2}\sum\limits_{i = 1}^m {{\big(\sum\limits_{j = 1}^n {{\mu^i_j}} {l_j}\big)^2}}\right)} {y^{{l_1}}}{y^{{l_2}}}\cdots{y^{{l_n}}}\equiv0.
$$
Now, it is obviously that for any nonzero coefficient $\Psi_{{l_1}\cdots{l_n}}$,
$$
\sum\limits_{i = 1}^n {{\lambda _i}{l_i}}  +\frac{1}{2} \sum\limits_{i = 1}^m {{\big(\sum\limits_{j = 1}^n {{\mu^i_j}} {l_j}\big)^2}}=0,
$$
equivalently,
\begin{align}
  \langle \lambda, \mathbf{l}\rangle+\frac{1}{2} \sum\limits_{i = 1}^m{|\langle \mu^i, \mathbf{l}\rangle|^2}=0,~~\mathbf{l}=(l_1,\cdots,l_n)\neq\mathbf{0},
\end{align}
which is in contradiction to the hypothesis of Theorem \ref{the3}.

\begin{remark}
When all the stochastic terms $g_i(x),i=1,\cdots,m$ vanish, the definition of local first integrals for SDEs $(\ref{e1})$ is consistent with that for ODEs, and Theorem \ref{the1} and Corollary \ref{coro} are reduced to Poincar\'e non-integrability Theorem \ref{poincare}.
In this sense, the above results generalize the classical Poincar\'e non-integrability theorem, Theorem A in \cite{kwek2003}, Theorem 1 in \cite{zhang2003} and Theorem 1 in \cite{cong2011jde}.
\end{remark}

\begin{remark}
Under the assumption $(\rm{H}1)$, we could get more refined results on
the number of functionally independent analytic strong first integrals in Theorem \ref{the1}. More specifically,
if $(\rm{H}1)$ holds, then the number of functionally independent analytic first integrals is less than or equal to the rank of
the set
\begin{align*}
\mathcal{S}=\{\mathbf{k}\in(\mathbb{Z}^+)^n: \langle \lambda, \mathbf{k} \rangle=0, \langle \mu^i, \mathbf{k} \rangle=0,i=1,\cdots,m,~~\mathbf{k}\neq0\},
\end{align*}
 which is obviously equal or lesser than $\mathrm{min}\{s_0, s_1,\cdots,s_m\}$.
\end{remark}

Note that the hypothesis of Theorem \ref{the3} is a little bit strict but suitable for our purpose. In order to facilitate the application, we point out
a special case of Theorem \ref{the3} where the stochastic terms $g_i(x),i=1,\cdots,m$ are higher-order with respect to $x$ in the next result. In this case, $Dg_i(0),i=1,\cdots,m$ are null matrices and diagonalizable, and we show that
Theorem \ref{the3} is also true without the diagonlization assumption of the matrix $Df(0)$.

\begin{thm}\label{the4}
Assume $f(0)=0$, $g_i(x)=O(|x|^2),i=1\cdots,m$. If the eigenvalues of $Df(0)$ do not satisfy any $\mathbb{Z^{+}}$-resonant condition,
then $(\ref{e1})$ doesn't  have any analytic weak first integral in a neighborhood of the origin.
\end{thm}

To prove Theorem \ref{the4}, we need the following result, which is due to Bibikov \cite{bibikov}.
\begin{lemma}\label{bibi}
Denote by $\mathcal{G}^r(\mathbb{C})$ the linear space of $n$-dimensional vector-valued homogeneous
polynomials of degree $r$ in $n$ variables with coefficients in the complex field $\mathbb{C}$. Let $A$ and $B$
be two $n\times n$ matrices with entries in $\mathbb{C}$, and their $n$-tuple of eigenvalues be $\lambda=(\lambda_1,\cdots,\lambda_n)$ and $\kappa=(\kappa_1,\cdots,\kappa_n)$, respectively.
Define a linear operator $L$ on $\mathcal{G}^r(\mathbb{C})$ as follows,
\begin{align*}
L[h](x) =\langle \partial_xh, Ax\rangle-Bh,~~h\in \mathcal{G}^r(\mathbb{C}).
\end{align*}
Then the spectrum of the operator $L$ is
\begin{align*}
\{\langle l,\lambda \rangle-\kappa_j;~~l\in(\mathbb{Z}^+)^n,|l|=r,j=1,\cdots,n\}.
\end{align*}
\end{lemma}
\end{proof}

\noindent{\bf Proof of Theorem \ref{the4}} Supposing (\ref{e1}) admits an analytic weak first integral $\Phi(x)$, by Theorem \ref{the1}, we see $\Phi(x)$ satisfies
\begin{equation}\label{eqv88}
\langle\nabla\Phi(x), f(x)\rangle+\frac{1}{2}\sum\limits_{i = 1}^m {g_i^{T}(x)\nabla^2\Phi(x) g_i(x)}\equiv 0.
\end{equation}
\noindent Expanding $\Phi(x)$ in the formal power series w.r.t $x$, we get
\begin{align}\label{f1}
\Phi(x)=\Phi_k(x)+\Phi_{k+1}(x)+\cdots,
\end{align}
where $k\geq1$, $\Phi_k(x)$ is the lowest order terms of $\Phi(x)$ and $\Phi_i(x)$ is the $i$-th order homogenous terms of $\Phi(x), i = k,k+1,\cdots$. Substitute (\ref{f1}) into (\ref{eqv88}) and equate the term in (\ref{eqv88}) of the lowest order with respect to $x$ yields
\begin{align}\label{lowest}
\langle\nabla\Phi_k(x),  Df(0)x\rangle\equiv0.
\end{align}
Consider the linear operator
$$
L[h](x)=\langle\nabla h(x),  Df(0)x\rangle
$$
on the linear space $\mathcal{H}_k$ of scalar homogeneous polynomials of degree $k$. By Lemma \ref{bibi}, the spectrum
of the operator $L$ over $\mathcal{H}_k$ is
$$
\bigg\{\sum_{i = 1}^n {{l_i}{\lambda _i}}\bigg|\sum_{i = 1}^n {{l_i} = k},l_i\in\mathbb{Z}^+ \bigg\}.
$$
\noindent Since the eigenvalues of the matrix $Df(0)$ do not satisfy any $\mathbb{Z}^+$-resonant condition,
zero does not belong to the spectrum
of the operator $L$ and so $L$ is an inverse operator over $\mathcal{H}_k$. It follows from (\ref{lowest})
that $\Phi_k(x)=0$, which is a contradiction. The proof is completed.

\begin{remark}\label{rrr}
Similar to Remark \ref{rr}, Theorem \ref{the4} are also true for the case of rational first integrals or Laurent polynomial first integrals of SDEs $(\ref{e1})$,
the major change being the substitution of $\mathbb{Z}^{+}$-resonant for $\mathbb{Z}$-resonant.
\end{remark}

At the last of this section, as an application, we turn to consider the effect of the stochastic noise on the integrability of differential equations.
More specifically, under nondegeneracy conditions, we show that for given ODEs, which maybe admit many first integrals, there exists a linear stochastic perturbation such that
the corresponding disturbed SDEs have no any analytic weak first integrals.
\begin{thm}\label{th5}
Let
\begin{align*}
\mathrm{d}x=f(x)\mathrm{d}t
\end{align*}
be an $n$-dimension analytic differential system such that $f(0)=0$ and $\mathrm{det}$$Df(0)\neq0$. Then there exists an $n\times n$ matrix $P$ such that
the disturbed SDEs
\begin{align*}
  \mathrm{d}x=f(x)\mathrm{d}t+Px  \mathrm{d}B_t
\end{align*}
have no any analytic weak first integrals, where $B_t$ is an one-dimensional Brown process.
\end{thm}

To prove Theorem \ref{th5}, we need the following lemma.
\begin{lemma}\label{lemma2}
If $\lambda_i\in\mathbb{C}, i=1,\cdots,n$ are all different from zero, then there exists $\mu_i\in\mathbb{C}, i=1,\cdots,n$ such that the  algebraic  equation with respect to ${\bf l}=(l_1,\cdots,l_n)$
\begin{align}\label{eeee}
2\sum\limits_{i=1}^n{\lambda_i l_i}+\sum\limits_{i=1}^n{l_i(l_i-1)u_i^2}+\sum\limits_{1\leq i\neq j\leq n}{l_il_j\mu_i\mu_j}=0
\end{align}
has no non-trivial solutions ${\bf l}\in(\mathbb{Z^{+}})^n/\{0\}$.
\end{lemma}

\begin{proof} Consider a field extension $M=\mathbb{Q}(\lambda_1,\cdots,\lambda_n)$ of rational number field $\mathbb{Q}$. Then since the transcendence degree of $\mathbb{C}$ over $\mathbb{Q}$ is the cardinality of the continuum, we see the transcendence degree of $\mathbb{C}$
over $M$ is also the cardinality of the continuum. So we may take $u$ as a transcendental element over $M$, and let $\mu_i=u^{a_i}, i=1, \cdots,m,$ where
$$
a_1=1, a_{k}=2(a_1+a_2+\cdots+a_{k-1}),~~k=2,\cdots,n.
$$
It is easy to check the elements $\mu_i^2=u^{2a_i}$, $\mu_i\mu_j=u^{a_i+a_j}$ are different from each other.

Assume $(l_1,\cdots,l_n)$ satisfies (\ref{eeee}) with $\mu_i=u^{a_i}$.
We will show that $l_1=\cdots=l_n=0$. In fact, by (\ref{eeee}), we obtain
\begin{align}
l_i(l_i-1)&=0,~~~~1\leq i\leq n\label{s2}\\
l_il_j&=0,~~~~1\leq i\neq j\leq n,\label{s3}\\
2\sum\limits_{i=1}^n{\lambda_i l_i}&=0.\label{s1}
\end{align}
By (\ref{s2}) and (\ref{s3}), we must have either all $l_i,i=1\cdots,n$ are zero, or one of them is $1$ and others are zero.
For the latter case, it follows
from (\ref{s1}) that one of $\lambda_i$ vanishes, which is inconsistent with the assumed conditions of $\prod\limits_{i = 1}^n {{\lambda _i}}\neq0$.
\end{proof}

\noindent\textbf{Proof of Theorem \ref{th5}.}\
Let $Q$ be a nonsingular matrix such that $Q^{-1}Df(0)Q = \Lambda$ is a Jordan canonical form, and $\mu=(\mu_1,\cdots,\mu_n)$ such that
(\ref{eeee}) has no any positive integral solution. We claim that
\begin{align}\label{sp}
\mathrm{d}x=f(x)\mathrm{d}t+Q\Lambda_1Q^{-1} x\mathrm{d}B_t
\end{align}
have no any analytic weak first integral, where $\Lambda_1=\mathrm{diag}(\mu_1,\cdots,\mu_n)$.

Suppose (\ref{sp}) has an analytic weak first integral $\Phi(x)$. Using the same argument as in the proof of Theorem \ref{the4}, we can easily carry out that the lowest homogeneous polynomial $\Phi_l(x)$ of $\Phi(x)$ satisfies
\begin{align}\label{sp2}
\langle\nabla_x\Phi_l, Df(0)x\rangle+\frac{1}{2}(\Lambda_1Q^{-1}x)^{T}\cdot\nabla_x^2\Phi_l \cdot(\Lambda_1Q^{-1}x)\equiv 0.
\end{align}
Under the change of the variables $x=Qy$, we have
\begin{align}\label{loow}
\nabla^2_{y}\Psi_l(y)=Q^{T}\nabla^2_{x}\Phi(x)Q,~~\nabla_{y}\Psi_l(y)=Q^T\nabla_{x}\Phi(x),
\end{align}
where $\Psi_l(y)=\Phi_l(Qy)$, and $\nabla^2_{y}\Psi$ and $\nabla^2_{x}\Phi$ are the Hessian matrix of $\Phi$ about $y$ and $x$, respectively.
Then, (\ref{sp2}) becomes
\begin{align}\label{sp3}
\langle\nabla_y\Psi_l, \Lambda y\rangle+\frac{1}{2}(\Lambda_1 y)^{T}\cdot\nabla_y^2\Psi_l \cdot(\Lambda_1y)\equiv 0.
\end{align}
Set
$$
\Lambda  = \left( {\begin{array}{*{20}{c}}
{{J_1}}&{}&{}&{}\\
{}&{J{}_2}&{}&{}\\
{}&{}& \ddots &{}\\
{}&{}&{}&{{J_m}}
\end{array}} \right),~~~
{J_r} = \left( {\begin{array}{*{20}{c}}
{{\lambda _r}}&1&{}&{}\\
{}&{{\lambda _r}}& \ddots &{}\\
{}&{}& \ddots &1\\
{}&{}&{}&{{\lambda _r}}
\end{array}} \right),
$$
where $J_r$ is a Jordon block with degree $i_r$.
Under the transformation of $y=Cz$ with
$$
C  = \left( {\begin{array}{*{20}{c}}
{{C_1}}&{}&{}&{}\\
{}&{C_2}&{}&{}\\
{}&{}& \ddots &{}\\
{}&{}&{}&{{C_m}}
\end{array}} \right),~~~
{C_r} = \left( {\begin{array}{*{20}{c}}
1&{}&{}&{}\\
{}&\varepsilon&{}&{}\\
{}&{}& \ddots &{}\\
{}&{}&{}&{{\varepsilon^{{i_r} - 1}}}
\end{array}} \right),
$$
we can rewrite (\ref{sp3}) as
\begin{align}\label{sp4}
\langle\nabla_{z}\Psi_l(Cz), C^{-1}\Lambda Cz\rangle+\frac{1}{2}(C^{-1}\Lambda_1 Cz)^{T}\cdot\nabla_{z}^2\Psi_l(Cz) \cdot(C^{-1}\Lambda_1 Cz)\equiv 0.
\end{align}

Since $\Psi_l(y)$ is a non-constant homogeneous polynomial of degree $l$ with respect to $y$, $\Psi_l(Cz)$ can be rewritten as
\begin{align}\label{lowest3}
\Psi_l(Cz)=\Psi^0_l(z)+\varepsilon\Psi^1_l(z)+\cdots+\varepsilon^{\tilde l}\Psi^{\tilde l}_l(z),
\end{align}
where $\Psi^i_l(z)$ are homogeneous polynomial of degree $l$ with respect to $z$. Furthermore, we have
\begin{align}\label{lowest4}
C^{-1}\Lambda_1 C=\Lambda_1,~~C^{-1}\Lambda C=B+\varepsilon \tilde B
\end{align}
with
\begin{align*}
B& = \left( {\begin{array}{*{20}{c}}
{{B_1}}&{}&{}&{}\\
{}&{{B_2}}&{}&{}\\
{}&{}& \ddots &{}\\
{}&{}&{}&{B{}_m}
\end{array}} \right),~~~~
{B_r} = \left( {\begin{array}{*{20}{c}}
{{\lambda _r}}&{}&{}&{}\\
{}&{{\lambda _r}}&{}&{}\\
{}&{}& \ddots &{}\\
{}&{}&{}&{{\lambda _r}}
\end{array}} \right),\\
\tilde B &= \left( {\begin{array}{*{20}{c}}
{{\tilde B_1}}&{}&{}&{}\\
{}&{{\tilde B_2}}&{}&{}\\
{}&{}& \ddots &{}\\
{}&{}&{}&{\tilde B{}_m}
\end{array}} \right),~~~~
{\tilde B_r} = \left( {\begin{array}{*{20}{c}}
0&1&{}&{}\\
{}&0& \ddots &{}\\
{}&{}& \ddots &1\\
{}&{}&{}&0
\end{array}} \right)
\end{align*}
Substituting (\ref{lowest3}) and (\ref{lowest4}) into (\ref{sp4}) and equating the terms of the lowest order with respect to $\varepsilon$ lead to
\begin{align}\label{sp8}
\langle\nabla_{z}\Psi_l^0(z),  Bz\rangle+\frac{1}{2}(\Lambda_1 z)^{T}\cdot\nabla_{z}^2\Psi_l^0(z) \cdot(\Lambda_1 z)\equiv 0.
\end{align}
In addition, expand $\Psi_l(z)$ as
\begin{align*}
\Psi^0_l(z)=\sum\limits_{{l_1} + ... + {l_n} =l } {{\Psi^0_{{l_1}...{l_n}}}z_1^{{l_1}}...z_n^{{l_n}}},~~\Psi_{{l_1}...{l_n}}\in\mathbb{C}.
\end{align*}
Then (\ref{sp8}) becomes
\begin{align}
\sum\limits_{{l_1} + ... + {l_n} =l } {{\Psi^0_{{l_1}...{l_n}}}\left(2\sum\limits_{i=1}^n{\lambda_i l_i}+\sum\limits_{i=1}^n{l_i(l_i-1)u_i^2}+\sum\limits_{1\leq i\neq j\leq n}{l_il_j\mu_i\mu_j}\right)z_1^{{l_1}}...z_n^{{l_n}}}\equiv0.
\end{align}
Since
$$
2\sum\limits_{i=1}^n{\lambda_i l_i}+\sum\limits_{i=1}^n{l_i(l_i-1)u_i^2}+\sum\limits_{1\leq i\neq j\leq n}{l_il_j\mu_i\mu_j}\neq0
$$
for any $l=(l_1,\cdots,l_n)\in (\mathbb{Z}^+)^n/\{0\}$, we obtain $\Psi^0_{{l_1}...{l_n}}=0$, which means $\Phi^0_l(z)$ is zero. This contradicts the fact $\Psi^0_l(z)$ is not a constant.

\section{Examples}\label{sec3}
In this section we give some examples to illustrate our
results.
\\
\noindent{\bf Example 1}
Consider the Sharma-Parthasarathy stochastic two-body equations
\begin{equation}\label{exa1}
 \left\{
  \begin{aligned}
   \mathrm{d}r=&v\mathrm{d}t\\
   \mathrm{d}\phi=&w\mathrm{d}t,\\
   \mathrm{d}v=&(rw^2-\frac{k}{mr^2})dt+r\sigma_{r}\mathrm{d}B_t^r,\\
   \mathrm{d}w=&-\frac{2vw}{r}\mathrm{d}t+\frac{\sigma_{\phi}}{r}\mathrm{d}B_t^{\phi},
  \end{aligned}
 \right.
\end{equation}
where $k,\sigma_{r},\sigma_{\phi}$ are positive constants, and $B_t^r, dB_t^{\phi}$ are two independent Brownian motions. This equations was introduced in \cite{exa1}
to model the force induced by a cloud having a density which fluctuates stochastically.
The observational data in \cite{exa11} about the zodiacal dust around the sun strongly support the validity of the model.

In this model, we denote by $x=(r,\phi,v,w)^T$, $f(x)=(v,w,rw^2-\frac{k}{mr^2},-\frac{2vw}{r})^T$, $g_1(x)=(0,0,r\sigma_{r},0)^T$ and $g_2(x)=(0,0,0,\frac{\sigma_{\phi}}{r})^T$. For the classical  deterministic two-body problem, there are two first integrals: the angular momentum
$M=mr^2w$ and the energy $E=\frac{1}{2}m(v^2+r^2w^2)-\frac{k}{r}$.
By Theorem \ref{the1}, neither $M$ nor $E$ is a first integral in the strong sense. However, by Theorem \ref{the1},
 the angular momentum $M=mr^2w$ still admits weak conservation property, namely, $M$ is a weak first integral.
\\
\\
\noindent{\bf Example 2}
As an example to illustrate Theorem \ref{the2}, let us consider system
\begin{align}
\left\{
\begin{aligned}
\mathrm{d}x_1 & =(ax_1+x_2x_3)\mathrm{d}t+(x_1-2x_2+x_1(x_2-x_3))\mathrm{d}B_t\\
\mathrm{d}x_2 & =(bx_2+x_2x_1-x_2x_3)\mathrm{d}t+(2x_2-x_3+x_2(x_3-x_1))\mathrm{d}B_t \label{ex} \\
\mathrm{d}x_3 & = (-ax_1-bx_2+x_2x_3)\mathrm{d}t+(x_3-x_1+x_3(x_1-x_2))\mathrm{d}B_t, \\
\end{aligned}
\right.
\end{align}
where $x$=$(x_1, x_2, x_3)\in \mathbb{C}^3$, and $B_t$ is an one-dimensional Brownian motion.

According to Theorem \ref{the1}, we can find $\Phi(x)=x_1+x_2+x_3$ is a strong first integral of SDEs (\ref{ex}).
In this case,
$$
Df(0) = \left( {\begin{array}{*{20}{c}}
a&0&0\\
0&b&0\\
{ - a}&{ - b}&0
\end{array}} \right),Dg(0) = \left( {\begin{array}{*{20}{c}}
1&{ - 2}&0\\
0&2&{ - 1}\\
{ - 1}&0&1
\end{array}} \right)
$$
and the eigenvalues of $Df(0)-\frac{1}{2}Dg(0)^2$ and $Dg(0)$ are
\begin{align*}
\lambda_1&=0,~~\lambda_2=\frac{1}{2}(\alpha+\sqrt{\beta}),~~
\lambda_3=\frac{1}{2}(\alpha-\sqrt{\beta}),\\
\mu_1&=0,~~\mu_2=2+i,~~\mu_3=2-i,
\end{align*}
where $\alpha=a+b-3$ and $\beta=(a-b)^2+8(a-b)-16$.

Obviously, the rank of the set
$$
\mathcal{S}_1:=\left\{(k_1,k_2,k_3)\in(\mathbb{Z}^+)^3:\mu_1k_1+\mu_2k_2+\mu_3k_3=0,k_1+k_2+k_3\neq0\right\}
$$
is equal to $1$. Hence, thanks to Theorem \ref{the2}, any analytic strong first integral of system (\ref{ex}) can be expanded a smooth function in $\Phi(x)$ if the rank of
$$
\mathcal{S}_0:=\left\{(k_1,k_2,k_3)\in(\mathbb{Z}^+)^3:\lambda_1k_1+\lambda_2k_2+\lambda_3k_3=0,k_1+k_2+k_3\neq0\right\}
$$
 is also equal to $1$, i.e.,
either
$\beta\neq0, \alpha/\sqrt{\beta}\notin\mathbb{Q}\cap(-1,1)$
or $\beta=0, \alpha\neq0$.
\\
\\
\noindent{\bf Example 3}
Consider the following $n$-dimensional SDEs of Lotka-Volterra type
\begin{align}\label{ex1}
 \mathrm{d}{x_i}(t) = {x_i}(t)\left( {{b_i} + \sum\limits_{j = 1}^n {{a_{ij}}{x_j}(t)} } \right)\mathrm{d}t + {x_i}\sum\limits_{j = 1}^n {{\sigma _{ij}}{x_j}(t)} \mathrm{d}B_t,~~ i=1,2,\cdot \cdot \cdot,n,
\end{align}
where $b_i \in \mathbb {R}$, $\sigma_{ij} \in \mathbb{R}$ and $B_t$ is an one-dimensional Brownian motion defined on the filtered space. This model was first derived in \cite{mao2003} as a stochastic extension of the fundamentally important population process, namely the Lotka-Volterra model. In \cite{mao2003}, Mao \emph{et al.} examined asymptotic behaviors of SDEs (\ref{ex1}). For more dynamical features of SDEs (\ref{ex1}) see \cite{zhu2009,du2006}.

According to Theorem \ref{the4} and Remark \ref{rrr}, if
$b_1,b_2,\cdots,b_n$ don't satisfy any $\mathbb{Z}^+-$resonant (or $\mathbb{Z}-$resonant) condition, then SDEs (\ref{ex1}) do not have any analytic weak first integral (or rational weak first integral).

\section{Discussion}

In this work, we have investigated the non-existence of local first integrals for stochastic differential equations, which can be regarded as an extension of the Poincare non-integrability theorem for ordinary differential equations.
In addition, there are three directions for our future work.
In one direction,
we will focus on the integrability and the normal form for SDEs.
We conjecture that the existence of $n-1$ functionally independent analytic strong first integrals
for  $\mathrm{d}X_t=\sum\limits_{i = 1}^m {{g_i}(X_t)\circ \mathrm{d}B^i_t}$
implies  there exists a random coordinate
transformation $X_t=H(\omega, Y_t)$ such that $H$ formally linearizes the above SDEs. The key point to deal with this problem is the relationship between the eigenvalues of $Dg_i(0)$ and the Lyapunov
exponents of the linear SDEs
 $\mathrm{d}X_t=\sum\limits_{i = 1}^m {{Dg_i(0)}(X_t)\circ \mathrm{d}B^i_t}.$
Next, we note that if the linear parts  of SDEs  has all its eigenvalues zero, then
the above results in this paper are trivial. For studying these cases we should consider semi-quasihomogeneous systems.
Hence, the other direction is to investigate the necessary conditions for the existence of  first integrals of SDEs (\ref{e1}) when $f$ and $g_i$ are semi-quasihomogeneous vector fields.
To this end, we think that the  definition
of Kowalevskaya
exponents should be introduced from ODEs to SDEs. Finally, our results in this paper studied the existence of functionally independent  first integrals of SDEs in a neighborhood of a singularity. The third direction
is to  investigate the existence of first integrals of SDEs in a neighborhood of a periodic orbit and the Floquet's theorem for SDEs should be studied.

\end{document}